\documentclass[11pt,reqno]{amsart}
\usepackage{amsfonts,amsmath}
\usepackage{amsthm}
\usepackage{graphicx,enumerate}
\usepackage[hmarginratio=1:1,hscale=0.75]{geometry}
\usepackage[latin1]{inputenc}
\usepackage[all]{xy}

\usepackage{amssymb,mathtools}
\usepackage{stmaryrd}
\usepackage{hyperref}
\usepackage{color}
\usepackage[T1]{fontenc}
\usepackage{lmodern}

\newtheorem{counter}{Counter}
\newtheorem{lem}[counter]{Lemma}
\newtheorem{defn}[counter]{Definition}
\newtheorem{thm}[counter]{Theorem}
\newtheorem{prop}[counter]{Proposition}
\newtheorem{cor}[counter]{Corollary}

\newtheorem{remark}[counter]{Remark}

\newcommand{\R}{\mathbb{R}}

\renewcommand{\L}{\mathcal{L}}

\renewcommand{\lg}{\langle}
\newcommand{\rg}{\rangle} 

\newcommand{\lra}{\longrightarrow}

 \newcommand{\sse}{\subseteq}

 \newcommand{\la}{\lambda}
 
\newcommand{\pd}{\partial}
\newcommand{\rhu}{{\overset{*}{\rightharpoonup}}}

\newcommand{\fal}{\forall}
\newcommand{\8}{\infty}

\newcommand{\vph}{\varphi}
\newcommand{\vep}{\varepsilon} 

 \newcommand{\om}{\Omega}

\newcommand{\gr}{\nabla}

 \renewcommand{\d}{\,\text{d}}

\newcommand{\TD}{\widetilde{V}}
\newcommand{\I}{\om}
\newcommand{\tdV}{\tilde{V}}

\DeclareMathOperator{\supp}{\textnormal{supp}}

\DeclareMathOperator{\argmin}{{\textnormal{argmin}}}

\newcommand{\red}{}
\newcommand{\blue}{}
\definecolor{mypur}{rgb}{0.42, 0,0.12}
\newcommand{\mypur}{}

\begin{document}

\title[Gradient flow approach to an exponential thin film equation]{Gradient flow approach to an exponential thin film equation: global existence and latent singularity}
\author{Yuan Gao}
\address{Department of Mathematics\\
   Hong Kong University of Science and Technology, Clear Water Bay, Kowloon, Hong Kong \&
   Department of Mathematics\\
   Duke University,
  Durham NC 27708, USA
   }
\email{maygao@ust.hk}
\author{Jian-Guo Liu}
\address{Department of Mathematics and Department of
  Physics\\Duke University,
  Durham NC 27708, USA}
\email{jliu@phy.duke.edu}
\author{Xin Yang Lu}
\address{
Department of Mathematical Sciences\\
Lakehead University, Thunder Bay ON P7B5E1, Canada \&
Department of Mathematics and Statistics
\\McGill University,
  Montreal QC H3A0B9, Canada}
\email{xlu8@lakeheadu.ca}

\begin{abstract}
  In this work, we study a fourth order exponential equation, $u_t=\Delta e^{-\Delta u},$ derived from thin film growth on crystal surface in multiple space dimensions. We use the gradient flow method in metric space to characterize the latent singularity in global strong solution, which is intrinsic due to high degeneration. We define a suitable functional, which reveals where the singularity happens, and then prove the variational inequality solution under very weak assumptions for initial data. Moreover, the existence of global strong solution is established with regular initial data.
\end{abstract}

\keywords{Fourth-order exponential parabolic equation, Radon measure, global strong solution, latent singularity, curve of maximal slope}
\subjclass[2010]{35K65, 35R06, 49J40}

\date{\today}

\maketitle

\section{Introduction}

\subsection{Background}
Thin film growth on crystal surface includes kinetic processes by which adatoms detach from above, diffuse on the substrate and then are absorbed at a new position.
These processes drive the morphological changes of crystal surface, which is related to various nanoscale phenomena \cite{SSR2, PimpinelliVillain:98}. Below the roughing temperature, crystal surfaces consist of facets and steps, which are interacting line defects.  At the macroscopic scale, the evoluion of
those interacting line defects is generally formulated as nonlinear PDEs using macroscopic variables; see \cite{Yip2001, our, Kohnbook, Margetis2006, Zang1990, Tang1997, Xiang2002}. Especially from rigorously mathematical level, \cite{Leoni2015, giga2010, our2, ourxu, She2011, LX} focus on the existence, long time behavior, singularity and self-similarity of solutions to various dynamic models under different regimes.

Let us first review the continuum model with respect to the surface height profile $u(t,x)$. Consider the general surface energy,
\begin{equation}
  G(u):=\int_{\Omega} {  (}\beta_1|\nabla u|+\frac{\beta_2}{p}|\nabla u|^p {  )}\d x,
\end{equation}
where $\Omega$ is the ``step locations area'' we {are} concerned {with}. Then the chemical potential $\mu$, defined as the change per atom in
the surface energy, can be expressed as
$$\mu:=\frac{\delta G}{\delta u}=-\nabla\cdot \Big( \beta_1 \frac{\nabla u}{|\nabla u|}+\beta_2|\nabla u|^{p-2}\nabla u\Big). $$

Now by conservation of mass, we write down the evolution equation for surface height of a solid film $u(t,x)$:
$$u_t+\nabla\cdot J=0,$$
where $$J=-M(\nabla u)\nabla\rho_s$$ is the adatom flux by Fick's law \cite{Margetis2006}, the mobility function $M(\nabla u)$ is a functional of the gradients
in $u$  and $ \rho_s$ is the local equilibrium
density of adatoms.
By the Gibbs-Thomson relation \cite{cooper1996, widom1982, Margetis2006}, which is connected to the theory of molecular capillarity, the corresponding local equilibrium
density of adatoms is given by
$$\rho_s=\rho^0 e^{\frac{\mu}{kT}},$$
where $\rho^0$
is a constant
reference density, $T$ is the temperature and $k$ is the Bolzmann constant.

Notice those parameters can be absorbed in the scaling
of the time or spatial variables. The evolution equation for $u$
can be rewritten as
\begin{equation}\label{eqn2}
u_t=\nabla\cdot \Big(M(\nabla u)\nabla e^{\frac{\delta G}{\delta u}}\Big).
\end{equation}

It should be pointed out that in past, the exponential
of $\mu/kT$ is typically linearized under the hypothesis that $|\mu|\ll kT$; see for instant \cite{D22, D24, Shenoy2002} and most rigorous results in \cite{Leoni2015, giga2010, our2, ourxu, She2011, LX} are established for linearized Gibbs-Thomson relation. This simplification, $e^\mu\approx 1+\mu$,
yields the linear Fick's law for the flux $J$ in terms of the  chemical potential
$$J=-M(\nabla u )\nabla \mu.$$
The resulting evolution equation is
\begin{equation}
\frac{\partial u}{\partial t} = \nabla \cdot \left(M(\nabla
u)\nabla\left(\frac{\delta G}{\delta u}\right)\right), \label{dynamicPDE}
\end{equation}
which is widely studied when
the mobility function $M(\nabla u)$ takes distinctive forms in different limiting
regimes. For example, in the diffusion-limited (DL) regime, where the
dynamics is dominated by the diffusion across the terraces and $M$ is
a constant $M \equiv 1$, Giga and Kohn \cite{giga2010} rigorously showed that with periodic
boundary conditions on $u$, finite-time flattening  occurs for $\beta_1\neq 0$.
A heuristic
argument provided by Kohn \cite{Kohnbook} indicates that the flattening
dynamics is
linear in time.  While in the attachment-detachment-limited (ADL) case, i.e.
the dominant processes are the attachment and detachment of atoms at step
edges and the mobility function \cite{Kohnbook} takes the form
$
M(\nabla u) = |\nabla u|^{-1},
$
we refer readers to \cite{Kohnbook, She2011, our2, ourxu} for analytical results.

Note that the simplifed version of PDE \eqref{dynamicPDE}, which  linearizes
 the Gibbs-Thomson relation, does not distinguish
between convex and concave parts of surface profiles. However the convex and concave parts of surface profiles actually have very different dynamic processes due to the exponential effect, which is explained in Section \ref{sec1.2} below; see also numerical simulations in \cite{LLDM}.

Now we consider the original exponential model \eqref{eqn2} in DL regime
 \begin{equation}\label{heq_exp}
\begin{aligned}
  u_t&=\nabla\cdot \Big( \nabla e^{\frac{\delta G}{\delta u}}\Big)
  =\Delta e^{ -\nabla\cdot \Big( |\nabla u|^{p-2}\nabla u\Big)},
  \end{aligned}
\end{equation}
with surface energy $G:=\int_\om \frac{1}{p}|\nabla u|^p\d x$, $p\geq 1$.
 The physical explanation of the $p$-Laplacian surface energy can be found in \cite{MarzuolaWeare2013}.
From the atomistic scale of solid-on-solid (SOS) model, the transitions between atomistic configurations are
determined by  the number of bonds that
each atom would be required to break in order to move.
It worth noting for $p=1$ \cite{LLDM}  developed an explicit
solution to characterize the dynamics of facet position in one dimensional, which is also verified by numerical simulation.

In this work, we focus on the case $p=2$ for high dimensional
 and use the gradient flow approach to study the strong solution with latent singularity to \eqref{heq_exp}. We will see clearly the different performs between convex and concave parts of the surface.
Explicitly,
given $T>0$ and a {\mypur bounded, spatial domain} $\om\sse \R^d$
with smooth boundary,
we consider the evolution problem
\begin{equation}\label{maineq}
\left\{
\begin{array}{cl}
u_t=\Delta e^{-\Delta u} & \text{in } \om\times [0,T],\\
\gr u\cdot \nu = \gr e^{-\Delta u} \cdot \nu =0& \text{on } \pd\om\times [0,T],\\
u(x,0) = u^0(x) & \text{on } \om,
\end{array}
\right.
\end{equation}
where $\nu$ denotes the outer unit normal vector to $\pd\om$.
The main results of this work is to prove the existence of variational inequality solution to \eqref{maineq} under weak assumptions for initial data and also the existence of strong solution  to \eqref{maineq} under strong assumptions for initial data; see Theorem \ref{th-vi} and Theorem \ref{mainth1} separately.

\subsection{Formal observations}\label{sec1.2}
We first show some a-priori estimates to see the mathematical structures of \eqref{maineq}.

On one hand, formally define the beam type free energy $F(u):=\int_\om e^{-\Delta u}\d x$, so we can rewrite the original equation as a gradient flow
  \begin{equation}\label{tm10_6}
    u_t=-\frac{\delta F}{\delta u}=\Delta e^{-\Delta u}
  \end{equation}
  and
  $$F(T)+\int_0^T\int_\om \big|\frac{\delta F}{\delta u}\big|^2 \d x\d t =F(0)$$
  for any $T>0$.

Notice boundary condition $\gr u\cdot \nu=0$. We have
$$\int_\om \Delta u \d x =0,$$
which  gives
\begin{equation}\label{925_01}
\|(\Delta u)^+\|_{L^1(\om)}=\|(\Delta u)^-\|_{L^1(\om)}=\frac{\|\Delta u\|_{L^1(\om)}}{2},
\end{equation}
where $(\Delta u)^+:=\max\{0, \Delta u\}$ is the positive part of $\Delta u$ and $(\Delta u)^-:=-\min\{0, \Delta u\}$ is the negative part of $\Delta u$.
Since
\begin{equation*}
\|(\Delta u)^-\|_{L^1} = \int_\om (\Delta u)^-\d x \le \int_\om e^{(\Delta u)^-}\d x \le \int_\om e^{-(\Delta u)^++(\Delta u)^-} \d x=  F(u) \le F(u^0)<+\8,
\end{equation*}
we know $\|\Delta u\|_{L^1(\om)}\le 2F(u^0)<+\8$. However, since $L^1$ is non-reflexive Banach space, the uniform bound of $L^1$ norm for $\Delta u$ dose not prevent it being a Radon measure. In fact, from $F(u)=\int_\om e^{-\Delta u} \d x$ and \eqref{tm10_6}, we can see a positive singularity in $\Delta u$ should be allowed for the dynamic model; also see an example in \cite[p.6]{LX} for a stationary solution with singularity.
We will introduce the latent singularity in $(\Delta u)^+$ officially in Section \ref{sec2.1}.

On the other hand, since
\begin{align*}
  \frac{1}{2}\frac{\d}{\d t}\int_\om u_t^2 \d x=\int_\om u_t (\Delta e^{-\Delta u})_t \d x=\int_\om \Delta u_t (e^{-\Delta u})_t \d x=\int_\om -(\Delta u_t)^2e^{-\Delta u}\d x\leq 0,
\end{align*}
we have high order a-priori estimate $$\int_\om u_t^2\d x=\int_\om (\Delta e^{-\Delta u})^2 \d x\leq C(u^0),$$
where $C(u^0)$ is a constant depending only on $u^0;$ see also \cite{LX}.
{\blue Noticing $F(t)=\int_{\Omega}e^{-\Delta u}\d x\leq C(u^0)$, from Poinc\'are's inequality, Young's inequality and the boundary condition $\gr e^{-\Delta u} \cdot \nu =0$, we have
\begin{equation}\label{tm28_01}
\begin{aligned}
\int_\om |e^{-\Delta u}|^2 \d x &\leq c\int_\om |\gr e^{-\Delta u}|^2 \d x +C(u^0) \\
&= c\int_\om - e^{-\Delta u} \Delta e^{-\Delta u}\d x +C(u^0)\\
&\leq \frac{1}{2}\int_\om |e^{-\Delta u}|^2 \d x+ c\int_\om |\Delta e^{-\Delta u}|^2 \d x+C(u^0),
\end{aligned}
\end{equation}
where $c$ is a general constant changing from line to line.
Hence we know
$$\int_\om |e^{-\Delta u}|^2 \d x \leq  c\int_\om |\Delta e^{-\Delta u}|^2 \d x+C(u^0).$$
}
Then by \cite[Lemma 1]{LX}, we have
\begin{equation}
\int_\om |D^2 e^{-\Delta u}|^2 \d x\leq c \int_\om(\Delta e^{-\Delta u})^2 \d x+C(u^0)\leq C(u^0).
\end{equation}
This, together with \eqref{tm28_01}, implies
\begin{equation}\label{0612_1}
\|e^{-\Delta u}\|_{H^2(\om)}\leq C_1(u^0).
\end{equation}
 Although these are formal observations for now, later we will prove them rigorously except for \eqref{0612_1}, which used formal boundary condition $\gr e^{-\Delta u} \cdot \nu =0$.

\subsection{Overview of our method and related method}
Although from formal observations in Section \ref{sec1.2} the original problem can be recast as a standard gradient flow, the main difficulty is how to characterize the latent singularity in $(\Delta u)^+$ and  choose a natural working space.

As we explained before, the possible existence of singular part for $\Delta u$ is intrinsic, so the best regularity we can expect for $\Delta u$ is Radon measure space. To get the uniform bound of $\|\Delta u\|_{\mathcal{M}(\om )}$, we need to first construct an invariant ball, which is the indicator functional $\psi$ defined in \eqref{psi}, then get rid of $\psi$ after we obtain the variational inequality solution; see Theorem \ref{th-vi} and Corollary \ref{cor10}. After we choose the working space $\mathcal{M}(\om )$ for $\Delta u$, we can define the energy functional $\phi$ rigorously in \eqref{phi} using Lebesgue decomposition. Using the gradient flow approach in metric space introduced by \cite{AGS}, we consider a curve of maximal slope of the energy functional $\phi+\psi$ and try to gain the evolution variational inequality (EVI) solution defined in Definition \ref{defweak} under weak assumptions for the initial data following \cite[Theorem~4.0.4]{AGS}. However, since the functional $\phi$ is defined only on the absolutely continuous part of $\Delta u$, it is not easy to verify the lower semi-continuity and convexity of $\phi$, which is developed in Section \ref{sec2.3} and Section \ref{sec2.4}. Finally, when the initial data have enough regularities, we prove the variational inequality solution has higher regularities and is also strong solution to \eqref{maineq} defined in Definition \ref{defstrong}.
{\mypur
We remark that the gradient flow in metric space is consistent with classical setting of gradient flow in Hilbert space. An alternative approach to study EVI solution is to use classical well-posednees theory for m-accretive operator in Hilbert Space; see for instant Theorem 3.1 in \cite{Brezis1973} or Theorem 4.5 in \cite{Barbu2010}.
However, to gain potential generalization to general energy, we ignore the Banach space structure and use the framework for gradient flow in metric space introduced by \cite{AGS}, which contains more understandings.
}

Recently, \cite{LX} also studies the same problem \eqref{maineq} using the method of approximating solutions. Their method based on carefully chosen regularization, which is delicate but the construction is subtle to reveal the mathematical structure of our problem. Instead, our method using gradient flow structure is natural and more general, which is flexible to wide classes of dynamic systems with latent singularity. When proving the variational inequality solution to \eqref{maineq}, we also provide an additional understanding for the evolution of thin film growth, i.e., the solution $u$ is a curve of maximal slope of the well-defined energy functional $\phi+\psi$; see Definition \ref{msdef}.

The rest of this work is devoted to first introduce the abstract setup of our problem in Section \ref{sec2.1} and Section \ref{sec2.2}. Then in Section \ref{sec2.3}, \ref{sec2.4} and Section \ref{sec2.5}, we prove the variational inequality solution following \cite[Theorem~4.0.4]{AGS}. In Section \ref{sec3}, under more assumptions on initial data, we finally obtain the strong solution to \eqref{maineq}.

\section{Gradient flow approach and variational inequality solution}
\subsection{Preliminaries}\label{sec2.1}
We first introduce the spaces we will work in. Since we are not expecting classical solution to \eqref{maineq}, the boundary condition in \eqref{maineq} can not be recovered exactly. Instead, we equip the boundary condition in the space $H, \TD$ defined blow.

Let
\begin{equation}\label{Hnote}
H:=\left\{u\in L^2(\om):\int_\om u\d x=0\right\},
\end{equation}
endowed with the standard scalar product $\lg u,v\rg_H:=\int_\om uv \d x$.

{ Since $L^1$ is not reflexive Banach space and has no weak compactness, those a-priori estimates in Section \ref{sec1.2} can not guarantee the $W^{2,1}(\om)$-regularity of solutions to \eqref{maineq}.  Hence we define the space $\TD$ as follows. {\blue Denote $\mathcal{M}$ as the space of finite signed Radon measures and $C_b(\om)$ is all the bounded continuous functions on $\om $. Denote $\|\cdot\|_{\mathcal{M}(\om)}$ the {total variation} of the measure. Take $d<p<\infty$, $\frac{1}{p}+\frac{1}{q}=1$. Define Banach space
\begin{equation}
\TD:= \{u\in H;\, \gr u\in L^q(\om), \,\Delta u\in \mathcal{M}(\om),\ \int_\om \vph\d(\Delta u  )=-\int_\om\gr u \cdot\gr \vph \d x \text{ for any }\vph\in W^{1,p}(\om)\}.
\end{equation}
   Endow $\TD  $
with the norm
\begin{equation}
\|u\|_{\TD  }:=\|u\|_{L^2(\om)}+\|\Delta u\|_{\mathcal{M}(\om)}.
\end{equation}
Next, we claim the norm is equivalent to $\|u\|_{L^2(\om)}+\|\gr u\|_{L^q(\om)}+\|\Delta u\|_{\mathcal{M}(\om)} $ by proving
\begin{equation}\label{eq11}
\|\gr u\|_{L^q(\om)}\leq c\|\Delta u\|_{\mathcal{M}(\om)}.
\end{equation}

Indeed, it is obvious when $d=1$ and we will prove it for $d \geq 2$. For $d<p<\infty$, $\frac{1}{p}+\frac{1}{q}=1$, we have $W^{1,p}(\om)\hookrightarrow C_b(\om).$ Noticing the Helmholtz-Weyl decomposition in \cite[Theorem III.1.2 and Lemma III.1.2]{HMNS}, we know
for any vector function $w\in L^p(\om)$ we have the Helmholtz-Weyl decomposition $w=\mathcal{P}w+\nabla \phi$ such that $\int_\om  \mathcal{P}w \cdot \gr v \d x=0$ for any $v\in W^{1,q}(\om)$, $\nabla \phi\in L^p(\om)$ and $\|\mathcal{P}w\|_{L^p}\leq C(p,\om)\|w\|_{L^p}.$ Hence for such $\phi$ and any $u\in \tilde{V}$, we know
\begin{equation}
\int_\om \phi d(\Delta u)= - \int_\om \gr \phi \cdot\gr u \d x =\int_\om (\mathcal{P}w-w)\cdot \gr u \d x=-\int_\om w \cdot \gr u \d x.
\end{equation}
Noticing also
$$\|\gr \phi\|_{L^p}\leq \|w\|_{L^p}+\|\mathcal{P}w\|_{L^p}\leq C(p,\om)\|w\|_{L^p},$$
 we can obtain \eqref{eq11} by
\begin{align*}
\|\gr u\|_{L^q}&\leq \sup_{w\in L^p} \frac{|\langle w,  \gr u\rangle|}{\|w\|_{L^p}}
=\sup_{w\in L^p} \frac{| \int_\om \phi d(\Delta u) |}{\|w\|_{L^p}}\\
&\leq \sup_{w\in L^p} \frac{ \|\phi\|_{L^\8}\|\Delta u\|_{\mathcal{M}} }{\|w\|_{L^p}}
\leq \sup_{w\in L^p} \frac{ \|\gr \phi\|_{L^p}\|\Delta u\|_{\mathcal{M}} }{\|w\|_{L^p}}\\
&\leq c  \|\Delta u\|_{\mathcal{M}}.
\end{align*}
}

Next, since $\Delta u$ can be a Radon measure, we need to make those formal observations in Section \ref{sec1.2} rigorous.
For any  $\mu\in \mathcal{M} ,$ from \cite[p.42]{evans1992}, we have the decomposition
\begin{equation}\label{decom}
  \mu=\mu_{\|} +\mu_{\bot}
\end{equation}
with respect to the Lebesgue measure,
where $\mu_{\|} \in L^1(\I)$ is the absolutely continuous part of $\mu$ and $\mu_{\bot}$ is the singular part, i.e., the support of $\mu_{\bot}$ has Lebesgue measure zero.
Define the beam type functional
\begin{equation}\label{phi}
\phi:H\lra [0,+\8],\qquad \phi(u):=
\begin{cases}
\int_\om e^{-(\Delta u)_{\|}^++(\Delta u)^-} \d x , & \text{if } u\in \TD {\blue \text{ and } (\Delta u)^-\ll\L^d,}\\
+\8 &\text{otherwise},
\end{cases}
\end{equation}
where $(\Delta u)_{\|}$ denotes the absolutely continuous part of $\Delta u$, {\blue $(\Delta u)^-$ is the negative part of $\Delta u$ and $(\Delta u)^+$ is the positive part of $\Delta u$ such that $(\Delta u)^{\pm} $ are two non-negative measures such that $\Delta u= (\Delta u)^+ -(\Delta u)^-.$} We call the singular part  $(\Delta u)^+_{\bot}$ latent singularity in solution $u$.
\begin{remark}
 Although the singularity vanishes in the energy functional $\phi$, it is not a removable singularity in the dynamics. Indeed, noticing the boundary condition,  we can not recover a new solution $v$ by removing the singularity such that $\Delta v=(\Delta u)^+_{\|}-(\Delta u)^-$ and $v_t=\Delta e^{-\Delta v}$. So the singularity in solution $(\Delta u)^+_{\bot}$ actually have effect on $u_t$ and we refer it  as latent singularity.
\end{remark}

{\mypur
An alternative definition and some useful properties for convex functional of measures can be found in \cite{DeTe1984, GoSe1964}.
We claim in the following lemma that the definition using duality for convex functional of measures is equivalent to our definition \eqref{phi} if $\Delta u$ is bounded from below. However, we only prove that $(\Delta u)^-\ll\mathcal{L}^d$ and do not have a lower bound for $\Delta u$. Therefore we prefer the current definition \eqref{phi}, which is defined only on the absolutely continuous part of $\Delta u.$

Recall the conjugate convex function of $f(x):=e^{-x}$ for $x\geq 0$ is
$$f^*(y)=\sup_{x\geq 0}(xy-f(x))=xy-f(x)\big|_{x=-\ln(-y)}=y-y \ln (-y) , \quad -1\leq y\leq 0.$$
Given some positive measure $\mu$, define the convex functional of $\mu$
\begin{equation}
\phi_1(\mu):=\sup_{-1\leq\vph\leq 0, \vph\in C_c^\8(\om)} \left\{\int_\om \vph \d \mu -\int_{\om} f^*(\vph) \d x \right\},
\end{equation}
where $f^*(y)=y-y\ln(-y), \,-1\leq y\leq 0$.

\begin{lem}\label{lemequi}
Assume $\mu\in \mathcal{M}^+(\om)$, $\mu_\|$ (resp. $\mu_\bot$) is the absolutely continuous part (resp. the singular part) of $\mu$ in decomposition \eqref{decom}. Denote
$\mu_\|=\rho \d x,$  $\om_+=\supp \mu_\bot$ and {$\om_-=\om\backslash\om_+$}. Then
\begin{equation}
\phi_1(\mu)=\int_\om e^{-\mu_\|}\d x.
\end{equation}
\end{lem}
\begin{proof}
From the definition of $\phi_1(\mu),$ we have
\begin{equation}\label{tt17}
\begin{aligned}
\phi_1(\mu)&=\sup_{-1\leq\vph\leq 0, \vph\in C_c^\8(\om)} \left\{\int_\om \vph \d \mu -\int_{\om} f^*(\vph) \d x \right\}\\
&=\sup_{-1\leq\vph\leq 0, \vph\in C_c^\8(\om)} \left\{\int_\om \vph \d \mu -\int_{\om} (\vph-\vph\ln(-\vph)) \d x \right\}\\
&=\sup_{-1\leq\vph\leq 0, \vph\in C_c^\8(\om)} \left\{\int_\om (-\vph +\vph\ln (-\vph)) \d x+\int_\om \vph \d\mu_\| +\int_\om \vph \d \mu_\bot \right\}\\
&=\sup_{-1\leq\vph\leq 0, \vph\in C_c^\8(\om)} \left\{\int_\om (-\vph +\vph\ln (-\vph)+\vph\rho) \d x +\int_\om \vph \d \mu_\bot \right\}.
\end{aligned}
\end{equation}

We claim
\begin{equation}\label{tt18}
\begin{aligned}
&\sup_{-1\leq\vph\leq 0, \vph\in C_c^\8(\om)} \left\{\int_\om \vph(\rho-1 +\ln(- \vph)) \d x +\int_\om \vph \d \mu_\bot \right\}\\
=& \sup_{-1\leq\vph\leq 0, \vph\in C_c^\8(\om), \atop\supp\vph \cap \om_+=\emptyset} \left\{\int_\om \vph(\rho-1 +\ln(- \vph)) \d x +\int_\om \vph \d \mu_\bot \right\}
\end{aligned}
\end{equation}
In fact, on one hand it is obvious that LHS of \eqref{tt18} $\geq$ RHS of \eqref{tt18}. On the other hand,
Since $-1\leq\vph\leq 0$ and $\mu_{\bot}\in \mathcal{M}^+(\om)$, we know $\int_\om \vph \d \mu_\bot\leq 0$ and $\int_\om \vph \d \mu_\bot =0$ for $\supp\vph \cap \om_+=\emptyset$. Hence
\begin{align*}
\text{LHS of \eqref{tt18}}  &\leq \sup_{-1\leq\vph\leq 0, \vph\in C_c^\8(\om)} \left\{\int_\om \vph(\rho-1 +\ln(- \vph)) \d x  \right\}.
\end{align*}
For any $\vep>0,$ there exists $-1\leq\vph_0\leq 0, \vph_0\in C_c^\8(\om)$ such that
\begin{align*}
\text{LHS of \eqref{tt18}}
&\leq \int_\om \vph_0(\rho-1 +\ln(- \vph_0)) \d x  +\vep .
\end{align*}
Notice $|\om_+|=0$. For $\vph_0$, from the strong Lusin's theorem \cite[p.8]{OTAM}, there exist compact set $K\subset\om_-$ and $f\in C_c^\8(\om_-)$ such that $f=\vph$ on $K$, $-1\leq f\leq 0$ and $\int_{\om\backslash K} (\rho+1) \d x\leq \vep$.
Hence we have
\begin{align*}
\text{LHS of \eqref{tt18}}
&\leq \Big(\int_{K}+\int_{\om\backslash K}\Big) \big(\vph_0(\rho-1 +\ln(- \vph_0))\big) \d x  +\vep \\
&\leq  \int_K f(\rho-1 +\ln(- f)) \d x +c\vep \\
&\leq \int_{\om_-} f(\rho-1 +\ln(- f)) \d x  +c\vep\\
&\leq \sup_{-1\leq f\leq 0, f\in C_c^\8(\om_-)} \left\{\int_{\om_-} f(\rho-1 +\ln(- f)) \d x  \right\}+c\vep\\
&=\sup_{-1\leq\vph\leq 0, \vph\in C_c^\8(\om), \atop\supp\vph \cap \om_+=\emptyset} \left\{\int_\om \vph(\rho-1 +\ln(- \vph)) \d x \right\}+c\vep,
\end{align*}
where the constant $c$ does not depends on $\vep$.
This implies LHS of \eqref{tt18} $\leq$ RHS of \eqref{tt18} $+c\vep$ and we know the claim \eqref{tt18} holds.

Combining \eqref{tt17} and \eqref{tt18}, we obtain
Therefore
\begin{align*}
\phi_1(\mu)
&=\sup_{-1\leq\vph\leq 0, \vph\in C_c^\8(\om_-)} \left\{\int_{\om_-} \vph(\rho-1 +\ln(- \vph)) \d x  \right\}\\
&= \int_{\om_-}\vph^*( \rho-1+\ln(- \vph^*) )\d x,
\end{align*}
where $\vph^*=-e^{-\rho}$ such that $F(\vph):=\int_{\om_-}\vph(\rho-1 +\ln(- \vph))\d x$,  $\frac{\delta F(\vph)}{\delta \vph}=0.$
Hence we have
\begin{equation}
\phi_1(\mu)=\int_{\om_-}\vph^*( 1-\ln \vph^*-  \rho )\d x=\int_{\om_-}e^{-\rho} \d x=\int_\om e^{-\mu_\|}\d x.
\end{equation}
\end{proof}
\begin{remark}
If $\Delta u\in \mathcal{M}^+(\om)$, taking $\mu=\Delta u$ in the definition \eqref{phi}, we can see from Lemma \ref{lemequi} that the two definitions are equivalent. If $\Delta u+C \in\mathcal{M}^+(\om)$, then we can take $\mu=\Delta u+C$ in Lemma \ref{lemequi} and definition \eqref{phi}.
\end{remark}
}

\medskip

In view of the a priori estimate on the mass of the measure $\Delta u$,
 we introduce the indicator function
\begin{equation}\label{psi}
\psi:H\lra \{0,+\8\},\qquad \psi(u):=
\begin{cases}
0 & \text{if } u\in \TD  ,\ \|\Delta u\|_{\mathcal{M}(\om)}\le C_*,\\
+\8 & \text{otherwise}.
\end{cases}
\end{equation}
Here $C_*$ is a fixed constant, which is determined in \eqref{consC} by the initial datum later. From Section \ref{sec1.2} we know the bound for $\|\Delta u\|_{\mathcal{M}(\om )}$  \eqref{psi}  is not artificial.

\subsection{Euler schemes}\label{sec2.2}
Even if \eqref{maineq} has a nice variational structure, and $V$ has Banach space structure,
the non-reflexivity of
$V$ imposes extra technical difficulties.
Instead of arguing with maximal monotone operator like in \cite{ourxu},
 we try to use the result \cite[Theorem~4.0.4]{AGS}
by Ambrosio, Gigli and Savar\'e. After defining the energy functional rigorously, we take the counterintuitive
approach of ignoring the variational structure of \eqref{maineq}
and the Banach space structure of $W^{2,1}(\om)$. In other words, we consider the gradient flow evolution in the {\em metric}
space $(H,\mbox{dist})$, with distance $\mbox{dist}(u,v):=\|u-v\|_H$.

Let $u^0\in {H}$ be a given initial datum and $0<\tau\ll1$ be a given parameter.
We consider a sequence $\{x_n^{\tau}\}$ which satisfies the following unconditional-stable backward Euler scheme
\begin{equation}\label{E}
\left\{
\begin{array}{l}
x^{(\tau)}_n\in \argmin_{x'\in H} \left\{(\phi+\psi)(x')+\dfrac1{2\tau} \|x'-x^{(\tau)}_{n-1}\|^2_{H}
\right\}  \qquad n\ge 1,\\
x^{(\tau)}_0:= u^0\in H.
\end{array}
\right.
\end{equation}
The existence and uniqueness of the sequence $\{x_n^{\tau}\}$ will be proved later in Proposition \ref{EulerTh}.
Thus we are considering the gradient descent with respect to $\phi+\psi$ in the space $(H,\mbox{dist})$.

Now for any $0<\tau\ll1$ we define the resolvent operator (see \cite[p. 40]{AGS})
\begin{equation*}
\mathcal{J}_\tau[u]:=\argmin_{v\in H} \left\{(\phi+\psi)(v)+\dfrac1{2\tau} \|v-u\|^2_{H}\right\},
\end{equation*}
then the variational approximation obtained by Euler scheme \eqref{E} is
\begin{equation}\label{Euler10}
u_n:=(\mathcal{J}_{t/n})^n[u^0].
\end{equation}
The results for gradient flow in metric space \cite[Theorem~4.0.4]{AGS} establish the convergence of the variational approximation $u_n$ to variational inequality solution to \eqref{maineq}, which is defined below.
\begin{defn}\label{defweak}
Given initial data $u^0\in H$, we call $u:[0,+\8)\lra H$ a variational inequality solution to \eqref{maineq} if $u(t)$ is a locally absolutely continuous curve such that $\lim_{t\to 0} u(t)=u^0$ in $H$ and
\begin{equation}
\frac12\frac\d{\d t}\|u(t)-v\|^2\le (\phi+\psi)(v)-(\phi+\psi)(u(t)), \quad \text{for a.e. } t>0,\,\forall v\in D(\phi+\psi).
\end{equation}
\end{defn}

Before proving the existence of variational inequality solution to \eqref{maineq}, we first study some properties of the functional $\phi+\psi$ in Section \ref{sec2.3} and Section \ref{sec2.4}.

\subsection{Weak-* lower semi-continuity for functional $\phi$ in $\TD$}\label{sec2.3}
{For any $\mu\in \mathcal{M}(\om)$, we denote $\mu\ll\L^d$ if $\mu$ is absolutely continuous with respect to Lebesgue measure and denote $\bar{\mu}:=\frac{\d\mu}{\d\L^d}$ as the density of $\mu$. For notational simplification, denote $\mu_{\|}$ (resp. $\mu_{\bot}$) as the absolutely continuous part (resp. singular part) of $\mu$ with respect to Lebesgue measure.

Let us first give the following proposition claiming weak-* lower semi-continuity for functional $\phi$ in $\TD$,  which will be used in Lemma \ref{convex}.
\begin{prop}\label{newlsc}
  Let $u_n,\, u\in \TD$. If $\Delta u_n {\rhu} \Delta u$ in $\mathcal{M}(\om)$, we have
  \begin{equation}\label{lsc01}
  \liminf_{n\to +\8}\phi(u_n)\geq \phi(u).
  \end{equation}
\end{prop}

Before proving Proposition \ref{newlsc}, we first state some lemmas.

From now on, we identify $\mu_n\ll\L^d$ with its density $\bar{\mu}_n:=\frac{\d\mu_n}{\d\L^d}$ and do not distinguish them for brevity.
Given $N>0$ and a sequence of measures $\mu_n$ such that $\mu_n\ll\L^d$,
 observe that
\begin{equation}\label{b}
\mu_n = \min\{\mu_n,N\}+\max\{\mu_n,N\}-N.
\end{equation}
{\blue
Let $\vph(\mu_n):=\int_\om e^{-(\mu_n)_\|}\d x$. First we state a lemma which shows that the uniform bound for $\vph(\mu_n)$ immediately rules out a negative singular part of $\mu$.
\begin{lem}\label{nnsp}
For any measure $\mu\ll \mathcal{L}^d$ any $N>0$, if $\vph(\mu)=\int_\om e^{-\mu}\d x\leq A<+\8$ for some bounded constant $A$, then we have the uniform estimate
\begin{equation}
\|\min\{\mu, N\}\|^2_{L^2(\om)}\leq 4e^N A+2|\om|N^2.
\end{equation}
\end{lem}
\begin{proof}
Noticing $e^{|x|}\geq \frac{x^2}{2}$ for any $x$, we have
\begin{align*}
&e^{-N}\int_\om |N-\min\{\mu,N\}|^2 \d x \\
=&  e^{-N}\int_{\{\mu\leq N\}} |N-\min\{\mu,N\}|^2 \d x\\
\leq & 2e^{-N}\int_{\{\mu\leq N\}} e^{N-\min\{\mu,N\}} \d x\\
=& 2\int_{\{\mu\leq N\}} e^{-\min\{\mu,N\}} \d x\\
=& 2\int_{\{\mu\leq N\}} e^{-\mu} \d x \leq 2A.
\end{align*}
Therefore we obtain
\begin{align*}
\|\min\{\mu, N\}\|^2_{L^2(\om)}\leq& \int_\om 2|N-\min\{\mu,N\}|^2 + 2N^2 \d x\\
\leq& 4e^N A+2|\om|N^2.
\end{align*}
\end{proof}
}
Next we state a lemma about the limit of the truncated measure $\min\{\mu_n,N\}.$
\begin{lem}\label{lemdunford}
For any $N>0$, given a sequence of measures $\mu_n$ such that $\mu_n\ll\L^d$, we assume moreover that $\mu_n\rhu\mu$ and $\vph(\mu_n)\leq A<+\8$ for some bounded constant $A$. Then there exist  measure $\mu_{\mbox{down}}\ll\L^d$ and  {\blue subsequence ($n_k$ still denoted as $n$) $\mu_n$,}  such that $N\ge \mu_{\mbox{down}}$ and $\min\{\mu_n,N\} \rhu\mu_{\mbox{down}}$.
\end{lem}
\begin{proof}
Since $\mu_n\rhu\mu$, we know there exists $\mu_{\mbox{down}}\in \mathcal{M}(\om)$ such that  $\min\{\mu_n,N\} \rhu\mu_{\mbox{down}}$ (upto subsequence). From $N-\min\{\mu_n,N\}\geq 0$ we have $N-\mu_{\mbox{down}}\geq 0$. Moreover we claim $\mu_{\mbox{down}}\ll\L^d$. {\blue From the assumption in Lemma \ref{lemdunford} we know $\vph(\mu_n)\leq A+1$ for all $n$. Therefore, from Lemma \ref{nnsp} we know $\|\min\{\mu_n, N\}\|^2_{L^2(\om)}\leq C(N,A).$ Hence $\mu_{\mbox{down}}\ll\L^d$. Moreover, from $N-\min\{\mu_n,N\}\geq 0$ we have $N-\mu_{\mbox{down}}\geq 0$.}
\end{proof}

 We also need the following useful lemma to clarify the relation between $\mu_{\mbox{down}}$ and the weak-$*$ limit of $\mu_n$.
\begin{lem}\label{lem0105}
 Given a sequence of measures $\mu_n$ such that $\mu_n\ll\L^d$,  we assume moreover that $\mu_n\rhu\mu$ and $\vph(\mu_n)\leq A<+\8$ for some bounded constant $A$. Then for any $N>0$, there exist $\mu_{\mbox{down}},\,\mu_{\mbox{up}}\in \mathcal{M}(\I)$ and {\blue subsequence ($n_k$ still denoted as $n$) $\mu_n$,} such that
 \begin{equation}\label{starlem1}
   \min\{\mu_n,N\}\rhu \mu_{\mbox{down}},\quad \mu_{\mbox{down}}\ll\L^d,\quad \mu_{\mbox{down}}\leq \mu_{\|},
 \end{equation}
  \begin{equation}\label{starlem2}
   \max\{\mu_n,N\}\rhu \mu_{\mbox{up}},\quad (\mu_{up})_\|\ge N,
 \end{equation}
 where $\mu_{\|}$ (resp. $\mu_{\bot}$) is the absolutely continuous part (resp. singular part) of $\mu$. Moreover,
 \begin{equation}\label{PHI}
   \int_{\om}e^{-\mu_{||}}\d x \leq \int_{\om} e^{-\mu_{{down}}}\d x.
 \end{equation}
\end{lem}
\begin{proof}
From Lemma \ref{lemdunford} we know, upon subsequence, $\min\{\mu_n,N\} \rhu\mu_{\mbox{down}}$ for some measure
$\mu_{\mbox{down}}$ satisfying $\mu_{\mbox{down}}\ll\L^d$ and $N\ge \mu_{\mbox{down}}$. By Lebesgue decomposition theorem, there exist unique measures $\mu_\|\ll\L^d$ and $\mu_\bot\bot\L^d$
such that $\mu=\mu_\|+\mu_\bot$.
The decomposition \eqref{b} then gives
$$0\le \mu_n-\min\{\mu_n,N\} = \max\{\mu_n,N\}-N \rhu\mu - \mu_{\mbox{down}}.$$
 Taking $\mu_{\mbox{up}}:=\mu - \mu_{\mbox{down}}+N$, as the sequence $\max\{\mu_n,N\}-N\ge0$, we obtain $\max\{\mu_n,N\}\rhu \mu_{\mbox{up}}$ and $(\mu - \mu_{\mbox{down}})_\|=\mu_{\mbox{up}\|}-N\ge 0$. Besides,  since $e^{-\mu_{\|}}$ is decreasing with respect to $\mu_{\|}$ and $\mu_{\|}\ge \mu_{\mbox{down}}$, we obtain \eqref{PHI}.
 \end{proof}

Now we can start to prove Proposition \ref{newlsc}.

\begin{proof}[Proof of Proposition \ref{newlsc}]
Assume $\Delta u_n{\rhu} \Delta u$ in $\mathcal{M}.$ Denote $f_n:=\Delta u_n$ and $f:=\Delta u$. {\blue  Set $L:= \liminf_{n\to +\infty}\phi(u_{n})$.  If $L= +\infty$ then $\eqref{lsc01}$ holds. If $L<\8$, which means there exists a subsequence such that $\lim_{k\to \8}\phi(u_{n_k})< +\infty$, then we take these subsequence (still denoted as $u_n$) and without loss of generality assume $\lim_{n\to\8}\phi(u_n)=L< +\infty$.  So $\phi(u_n)\leq L+1$ for all large $n$ and $f_n^-\ll\L^d$.}

 {\blue Since $\phi$ is defined only on the regular part of $\Delta u$, we concern about the ``cross convergence" case. In fact,  by the convexity of $\varphi(v):=\int_\om e^{-v}\d x$ on $L^1(\om)$ and \cite[Corollary 3.9]{BrezisF}, we know $\varphi(v)$ is l.s.c on $L^1(\om)$ with respect to the weak topology. Therefore, if we have $f_{n\|}\rhu f_{\|}$ and $f_{n\bot}\rhu f_{\bot}$, then \eqref{lsc01} holds.
This implies that we only need to consider two ``cross convergence" cases:} (i) there are some $f_n$ are positive measures, i.e. $f_{n\bot}\neq 0$, and $f_{n\|}\rhu g_1\ll\L^d$, $f_{n\bot}\rhu g_2\ge 0$ and $g_1+g_2=f_{\|}$; or (ii) all $f_n$ are absolutely continuous and $f_{n\|}=f_n$ may weakly-* converge to a singular measure.

 For case (i), if we have $f_{n\|}\rhu g_1\ll\L^d$, $f_{n\bot}\rhu g_2\ge 0$ and $g_1+g_2=f_{\|}$, then since $e^{-f_{\|}}$ is decreasing with respect to $f_{\|}$, we have $\int_\om e^{-g_1}\d x\geq \int_\om e^{-f_{\|}}\d x$.
On the other hand, we know $\varphi(v):=\int_\om e^{-v}\d x$ is lower-semicontinuous on $L^1(\om)$
with respect to the strong topology.
Hence by the convexity of $\varphi(v):=\int_\om e^{-v}\d x$ on $L^1(\om)$ and \cite[Corollary 3.9]{BrezisF}, we know $\varphi(v)$ is l.s.c on $L^1(\om)$ with respect to the weak topology. So $f_{n\|}\rhu g_1\ll\L^d$ gives $f_{n\|}\rightharpoonup g_1 $ in $L^1(\om)$ and
\begin{equation}\label{lsc02}
\liminf_n\phi(u_n)=\liminf_{n}\int_\om e^{-f_{n\|}}\d x\geq\int_\om e^{-g_1}\d x\geq \int_\om e^{-f_{\|}}\d x=\phi(u)
\end{equation}
which ensure \eqref{lsc01} holds.

  Now we concern the case (ii): $f_{n\bot}=0$ and $f_{n\|}=f_n$ may weakly-* converge to a singular measure. {\blue First from $\phi(u_n)\leq L+1$ and Lemma \ref{nnsp}, we know $f^-\ll \mathcal{L}^d.$ }
  For any $N>0$ large enough,
denote $\phi_{N}(u_n):=\int_\I  e^{-\min\{f_n,N\}}\d x.$
Then
the truncated measures $\min\{f_n,N\}$ satisfy
\begin{align*}
 \phi_N(u_n) &=\int_\I  e^{-\min\{f_n,N\}}\d x\\
 & = \int_{\{f_n\le N\}}e^{-\min\{f_n,N\}}\d x+e^{-N}\L^d(\{f_n> N\})\\
&\ge\int_{\{f_n\le N\}}e^{-f_n}\d x+\int_{\{f_n> N\}}e^{-f_n}\d x= \phi(u_n).
\end{align*}
The second equality also shows
\begin{align*}
  \phi_N(u_n)-e^{-N}\L^d(\{f_n> N\})&=\int_{\{f_n\le N\}}e^{-\min\{f_n,N\}}\d x\\
  &\leq \int_\I  e^{-f_n} \d x =\phi(u_n).
\end{align*}
Hence we obtain
\begin{equation}\label{0104_28o}
  |\phi(u_n)-\phi_N(u_n)|\leq e^{-N}\L^d(\{f_n> N\})\leq e^{-N}|\om|.
\end{equation}

From Lemma \ref{lem0105}, we know the truncated sequence $\min\{f_n,N\}$ satisfies
\begin{equation}\label{star0105_1}
   \min\{f_n,N\}\rhu f_{\mbox{down}},\quad f_{\mbox{down}}\ll\L^d,\quad
   \int_\I  e^{-f_{\mbox{down}}}\d x\geq \int_\I  e^{-f_\|}\d x.
 \end{equation}
Since $\min\{f_n,N\}\rightharpoonup f_{\mbox{down}}$ in $L^1(\om)$, using the same argument with \eqref{lsc02}, we obtain
\begin{equation}\label{conv}
\liminf_{n\to+\8}\int_\I  e^{-\min\{f_n,N\}}\d x\ge \int_\I  e^{-f_{{down}}}\d x\geq \int_\I  e^{-f_\|}\d x=\phi(u).
\end{equation}
Combining this with \eqref{0104_28o}, we obtain
\begin{equation}
\begin{aligned}
  \liminf_{n\to+\8}\phi(u_n)&\geq \liminf_{n\to+\8} \phi_N(u_n)-e^{-N}|\om|\\
  &=\liminf_{n\to+\8}\int_\I  e^{-\min\{f_n,N\}}\d x-e^{-N}|\om|\\
  &\geq \phi(u)-e^{-N}|\om|,
  \end{aligned}
\end{equation}
and thus we complete the proof of Proposition \ref{newlsc} by the arbitrariness of $N$.
 \end{proof}
}

\subsection{Convexity and lower semi continuity of functional $\phi+\psi$ in $H$}\label{sec2.4}

\begin{lem}\label{convex}
The sum $\phi+\psi:H\lra [0,+\8]$ is proper, convex, lower semicontinuous in $H$ and satisfies coercivity defined in \cite[(2.4.10)]{AGS}.
\end{lem}

\begin{proof}
Clearly since $u\equiv 0\in D(\phi+\psi)$, $D(\phi+\psi)=\{\phi+\psi<+\8\}$ is non empty, hence $\phi+\psi$ is proper.
Due to the positivity of $\phi,\, \psi$, coercivity \cite[(2.4.10)]{AGS}, {\red i.e.,
$\exists u*\in D(\phi+\psi), r*>0 \text{ such that } \inf\{(\phi+\psi) (v): v\in H, \text{dist}(v,u*)\leq r*\}>-\infty,$
} is trivial.
\medskip

{\em Convexity.} Note that since both $\phi$, $\psi\ge 0$, we have $D(\phi+\psi)=D(\phi)\cap D(\psi)$.
Given $u,v\in H$, $t\in (0,1)$, without loss of generality assume $u,v\in D(\phi+\psi)$,
otherwise convexity inequality is trivial. Thus $(1-t)u+tv\in D(\psi)$, and
the measure $\Delta[(1-t)u+tv]$ has no negative singular part, while its positive
singular part satisfies
$$(\Delta[(1-t)u+tv])^+_\bot=(1-t)(\Delta u)_\bot^+ +t(\Delta v)^+_\bot,$$
and its absolutely continuous part satisfies
$$(\Delta[(1-t)u+tv])_\|=(1-t)(\Delta u)_\| +t(\Delta v)_\|.$$
Thus
\begin{align*}
\phi((1-t)u+tv) & =\int_\om e^{-[(1-t)\Delta u+t\Delta v]_\|}\d x=\int_\om e^{-[(1-t)(\Delta u)_\|+t(\Delta v)_\|]}\d x\\
&\le
\int_\om [(1-t)e^{-(\Delta u)_\|}+te^{-(\Delta v)_\|}]\d x\\
&= (1-t)\phi(u)+t\phi(v),
\end{align*}
hence $\phi+\psi$ is convex.

\medskip

{\em Lower semicontinuity.} Consider a sequence $u_n\to u$ in $H$. We need to check
$$(\phi+\psi)(u)\le \liminf_n (\phi+\psi)(u_n).$$
{\red If $u_n\in D(\phi+\psi)$ does not hold}
for all large $n$, then lower semicontinuity is trivial.
Without loss of generality, we can assume $u_n\in D(\phi+\psi)$ for all $n$, and also
$$\liminf_n (\phi+\psi)(u_n)=\lim_n (\phi+\psi)(u_n).$$
Since $u_n\in D(\psi)$, we have $\|\Delta u_n\|_{\mathcal{M}(\om)}\le C_*$, hence there exists
$v\in \mathcal{M}(\om)$ such that $\Delta u_n\rhu v$. Since we also have $u_n\to u$ in $H$ so $v= \Delta u$ and we know $\|\Delta u\|_{\mathcal{M}(\om)}\le C_*$. {\blue From \eqref{eq11} we also know $u\in \tilde{V}$.}
Then $0=\psi(u_n)=\psi(u)$ and by Proposition \ref{newlsc}, we have
\begin{equation*}
\liminf_n \phi(u_n) \geq \phi(u)
\end{equation*}
so the lower semicontinuity is proved.
\end{proof}

\begin{lem}[$\tau^{-1}$-convexity]\label{convex2}
For any $u,v_0,v_1\in D(\phi+\psi)$, there exists a curve $v:[0,1]\lra D(\phi+\psi)$
such that $v(0)=v_0,\, v(1)=v_1$ and
the functional
\begin{equation}\label{PhiD}
\Phi(\tau,u;v):=(\phi+\psi)(v)+\frac1{2\tau} \|u-v\|_H^2
\end{equation}
satisfies
\begin{equation}\label{c}
\Phi(\tau,u;v(t))\le (1-t)\Phi(\tau,u;v_0)+t\Phi(\tau,u;v_1)-\frac1{2\tau}t(1-t)\|v_0-v_1\|_H^2
\end{equation}
for all $\tau>0$.
\end{lem}
We remark that \eqref{c} is the so-called ``$\tau^{-1}$-convexity'' \cite[Assumption~4.0.1]{AGS}.

\begin{proof}
Let $v(t):=(1-t)v_0+tv_1$.
The proof follows from the simple identity
\begin{equation*}
\|(1-t)v_0+tv_1-u\|_H^2 = (1-t)\|u-v_0\|_H^2+t\|u-v_1\|_H^2- t(1-t)\|v_0-v_1\|_H^2.
\end{equation*}
The convexity of $\phi+\psi$ then gives
\begin{align*}
\Phi(\tau,u;v(t)) &= (\phi+\psi)((1-t)v_0+tv_1)+\frac1{2\tau} \|u-[(1-t)v_0+tv_1]\|_H^2 \\
&\le (1-t)(\phi+\psi)(v_0)+t(\phi+\psi)(v_1)\\
&\qquad +\frac1{2\tau}(1-t)\|u-v_0\|_H^2+\frac1{2\tau}t\|u-v_1\|_H^2-\frac1{2\tau}t(1-t)\|v_0-v_1\|_H^2\\
&=(1-t)\Phi(\tau,u;v_0)+t\Phi(\tau,u;v_1)-\frac1{2\tau}t(1-t)\|v_0-v_1\|_H^2,
\end{align*}
and concludes the proof.
\end{proof}

After above properties for functional $\phi+\psi$, we state existence and uniqueness of the sequence $\{x_n^{\tau}\}$ chosen by Euler scheme
\eqref{E}.
\begin{prop}\label{EulerTh}
  Given parameter $\tau>0,$ $u^0\in H,$ then for any $n\geq 1$, there exists unique $x_n^{\tau}$ satisfying \eqref{E}.
\end{prop}
\begin{proof}
  Given $n\geq 1$, we will prove this proposition by the direct method in calculus of variation. Let $\Phi(\tau, x_{n-1}; x)$ defined in \eqref{PhiD} and $A:=\inf_{x\in H} \Phi(\tau, x_{n-1}; x)$. Then there exist $\{x_{n_i}\}\sse D(\Phi)$ such that $\Phi(\tau, x_{n-1}; x_{n_i})\to A$ as $i\to +\8$ and $\Phi(\tau, x_{n-1}; x_{n_i})$ are uniformly bounded. Hence upon a subsequence, there exists $x_n\in H$ such that $x_{n_i}\rightharpoonup x_n$ in $H$. This, together with the uniform boundedness of $\|\Delta x_{n_i}\|_{\mathcal{M}(\om)}$ shows that $\Delta x_{n_i}\rhu v= \Delta x_n$ in $\mathcal{M}(\om).$ Then by Proposition \ref{newlsc} we have
  $$A=\liminf_{i\to +\8}\Phi(\tau, x_{n-1}; x_{n_i})\geq \Phi(\tau, x_{n-1}; x_n)\geq A,$$
  which gives the existence of $x_n$ satisfying \eqref{E}.

  The uniqueness of $x_n$ follows obviously by the convexity of $\phi$ and the strong convexity of $\|\cdot\|_H.$

\end{proof}

\subsection{Existence of variational inequality solution}\label{sec2.5}

After those preparations in Section \ref{sec2.3} and Section \ref{sec2.4}, in this section we apply the convergence result in \cite[Theorem~4.0.4]{AGS} to derive that the discrete solution $u_n$ obtained by Euler scheme \eqref{E} converges to the variational inequality solution defined in Definition \ref{defweak}.
 For $v\in D(f)$, denote the local slope
\begin{equation}\label{localslope}
 |\pd f|(v):=\limsup_{w\to v}\frac{\max\{f(v)-f(w),0\}}{\mbox{dist}(v,w)}.
\end{equation}
 {\blue Take $f=\phi+\psi$, by the $\tau^{-1}$-convexity in Lemma \ref{convex2}} and \cite[Theorem~2.4.9]{AGS} for $\la=0$, the local slope coincides with the global slope
 $$\iota_f(v):=\sup_{v\neq w} \frac{\max\{f(v)-f(w),0\}}{\|v-w\|_H},$$
  i.e.
\begin{equation}\label{iota}
|\pd f|(v)=\iota_f(v).
\end{equation}

We point out that with Lemma \ref{convex} and \cite[Theorem~1.2.5]{AGS}, we also know the global slope $\iota_f$ is a strong upper gradient for {\blue $f=\phi+\psi$}. Hence for $\iota_f$, we recall  \cite[Definition~1.3.2]{AGS} for curves of maximal slope.
\begin{defn}\label{msdef}
  Given a functional $f:\,D(\phi)\to \mathbb{R}$ and the global slope $\iota_f$, we say that a locally absolutely continuous map $u:\,(0,T)\to H$ is a curve of maximal slope for the functional $f$ with respect to $\iota_f$ if
  \begin{equation}
    (f(u(t)))'\leq -\frac{1}{2} |u_t|^2-\frac{1}{2}\iota_f(u)^2 \text{ for a.e. }t\in(0,T).
  \end{equation}
\end{defn}

Now the hypotheses of \cite[Theorem~4.0.4]{AGS} are all satisfied: Lemma \ref{convex} gives convexity, lower semicontinuity and coercivity of $\phi+\psi$
\cite[(4.0.1)]{AGS}, while Lemma \ref{convex2} gives $\tau^{-1}$-convexity of $\phi+\psi$ with $\la=0$ \cite[Assumption~4.0.1]{AGS}.
Thus we have:

\begin{thm}\label{th-vi}
Given $u^0 \in H$,
\begin{enumerate}[(i)]
  \item (convergence and error estimate) for any $t>0$, $t=n\tau$, let $u_n$ in \eqref{Euler10} be the solution obtained by Euler scheme \eqref{E}, then there exists a local Lipschitz curve $u(t):[0,+\8)\to H$ such that
      \begin{equation}\label{tmp23}
        u_n\to u(t)\text{  in }L^2(\om)
      \end{equation}
      and if further $\phi(u^0)<+\8$, we have the error estimate
      \begin{equation}
        \|u(t)-u_n\|_{H}\leq \frac{\tau}{\sqrt{2}}|\pd \phi|(u^0) ;
      \end{equation}
  \item $u:[0,+\8)\lra H$ is the unique EVI solution to \eqref{maineq}, i.e., $u$ is unique among all the locally absolutely continuous curves such that $\lim_{t\to 0} u(t)=u^0$ in $H$ and
\begin{equation}\label{vi}
\frac12\frac\d{\d t}\|u(t)-v\|^2\le (\phi+\psi)(v)-(\phi+\psi)(u(t)), \quad \text{for a.e. } t>0,\,\forall v\in D(\phi+\psi);
\end{equation}
  \item $u(t)$ is a locally Lipschitz curve of maximal slope of $\phi$ for $t>0$ in the sense
  \begin{equation}\label{ms0427}
    \big((\phi+\psi)(u(t))\big)'\leq -\frac{1}{2} |u_t|^2-\frac{1}{2}\iota_\phi(u)^2;
  \end{equation}
  \item moreover, we have the following regularities
\begin{align}
(\phi+\psi)(u(t)) &\le (\phi+\psi)(v)+\frac1{2t}\|v-u^0\|_H^2 \qquad \fal v\in D(\phi+\psi),\label{phi-dec}\\
|\pd(\phi+\psi)|^2(u(t)) &\le |\pd(\phi+\psi)|^2(v)+\frac1{t^2}\|v-u^0\|_H^2\qquad\fal v\in D(|\pd(\phi+\psi)|),\label{Dphi-dec}
\end{align}
\begin{equation}\label{asym}
|\pd(\phi+\psi)|(u(t)) \le \frac{\|u^0-\bar u\|_H}t,\qquad (\phi+\psi)(u(t))-(\phi+\psi)(\bar u)\le \frac{\|u^0-\bar u\|_H^2}{2t},
\end{equation}
and $t\mapsto \|u(t)-\bar u\|_H$ is non-increasing, where $\bar u$ is a minimum point for $\phi+\psi$ and
$|\pd( \phi+\psi)|(v)=\iota_{\phi+\psi}(v)$
 is the local slope;
\item ($L^2$-contraction) let $u^0,\, v^0\in H$ and $u(t),\, v(t)$ be solutions to the variational inequality \eqref{vi}, then
    \begin{equation}
      \|u(t)-v(t)\|_H\leq \|u^0-v^0\|_H.
    \end{equation}
\end{enumerate}
\end{thm}

\begin{proof}
Since from Lemma \ref{convex} and Lemma \ref{convex2}, we are under the hypotheses of \cite[Theorem~4.0.4]{AGS}, we apply it with energy functional
$\phi+\psi$, and metric space $(H,\mbox{dist})$, $\mbox{dist}(u,v)=\|u-v\|_H$ to obtain \eqref{tmp23}.
{\blue
Notice the Assumption in \cite[Theorem~4.0.4]{AGS} requires $u^0 \in \overline{D(\phi+\psi)}^{\|\cdot\|_H}$. We notice that $u^0\in D(\phi+\psi)$ means (a) $\phi(u^0)<+\8$ and (b) $\psi(u^0)<+\8$. From the definition \eqref{phi} we know (a) requires $u^0\in \tilde{V}$, $(\Delta u^0)^-\ll\mathcal{L}^d$ and $\int_\om e^{-(\Delta u^0)^+_\|+(\Delta u^0)^-}\d x<+\8.$ Similar to the discussion for \eqref{psi00} we also know (a) implies (b) for $C_*=2\phi(u^0)+1$ in \eqref{consC}. Therefore, $u^0\in D(\phi+\psi)$ if and only if $\phi(u^0)<+\8$, i.e., $u^0\in \tilde{V}$, $(\Delta u^0)^-\ll\mathcal{L}^d$ and $\int_\om e^{-(\Delta u^0)^+_\|+(\Delta u^0)^-}\d x<+\8.$ Since $W^{2,\8}(\om)$ is dense in $H$, we also know $\overline{D(\phi+\psi)}^{\|\cdot\|_H}=H.$
}

Therefore the convergence result (i) comes from \cite[(4.0.11),(4.0.15)]{AGS}.
The variational inequality
\eqref{vi} follows from \cite[(4.0.13)]{AGS}. \cite[Theorem 4.0.4 (ii)]{AGS} shows the result (iii) and \eqref{ms0427} follows  Definition \ref{msdef} of maximal slope.

Regularities \eqref{phi-dec} and
 \eqref{Dphi-dec} follow from \cite[(4.0.12)]{AGS}.
 Asymptotic behavior
 \eqref{asym} and monotonicity of $t\mapsto \|u(t)-\bar u\|_H$ follow from
 \cite[Corollary~4.0.6]{AGS}, which requires the same hypotheses of \cite[Theorem~4.0.4]{AGS}. Finally, the contraction result (v) follows from \cite[(4.0.14)]{AGS}.
\end{proof}

 \section{{ Strong solution}}\label{sec3}

 We will prove the variational inequality solution obtain in Theorem \ref{th-vi} is actually a strong solution in this section.

Now we assume
$u:[0,+\8)\lra H$ is the unique solution of EVI \eqref{vi}, i.e.,
\begin{equation}\label{0927-32}
\frac12\frac\d{\d t}\|u(t)-v\|^2\le (\phi+\psi)(v)-(\phi+\psi)(u(t)), \quad \text{for a.e. } t>0,\,\forall v\in D(\phi+\psi).
\end{equation}

\subsection{Regularity of variational inequality solution}\label{Sec3.1}
First we state EVI solution has further regularities.
\begin{cor}\label{cor10}
Given  $T>0$ and initial datum $u^0 \in H$ such that $\phi(u^0)<+\8$, the solution obtained in Theorem \ref{th-vi} has the following regularities
 $$u\in L^\8([0,T];\tdV )\cap C^0([0,T];H),\quad u_t\in L^\8([0,T];H),$$
 $$(\Delta u)^-\ll\L^d\quad \text{ for a.e. }t\in[0,T],$$
 where
 $(\Delta u)^-$ is the negative part of $\Delta u$.
 Besides, we can rewrite EVI \eqref{vi} as
 \begin{equation}\label{vi35}
\lg u_t(t),u(t)-v\rg_{{ H',H}}\le \phi(v)-\phi(u(t)) \quad \text{for a.e. } t>0,\,\forall v\in {\blue D(\phi+\psi)}.
\end{equation}
\end{cor}
{\red
The dual pair $\lg \cdot, \cdot \rg_{H',H}$ is the usual integration so we just use $\lg \cdot, \cdot \rg$ in the following article.}
 Recall the definition of $\phi$ in \eqref{phi}. $\phi(u^0)<+\8$ if and only if $u^0\in \tilde{V}$, $(\Delta u^0)^-\ll\mathcal{L}^d$ and $\int_\om e^{-(\Delta u^0)^+_\|+(\Delta u^0)^-}\d x<+\8.$
}
 \begin{proof}
 First,
we claim the functional $\psi$ can be taken off. Indeed, from \eqref{phi-dec} we have
\begin{equation}
(\phi+\psi)(u(t)) \le (\phi+\psi)(v)+\frac1{2t}\|v-u^0\|_H^2 \qquad \fal v\in \blue D(\phi+\psi).
\end{equation}
Then taking $v=u^0$ gives
\begin{equation}\label{phi0}
(\phi+\psi)(u(t))\le (\phi+\psi)(u^0)<+\8,
\end{equation}
which also implies
\begin{equation}\label{phi0n}
\phi(u(t))\le \phi(u^0)<+\8 \quad \text{  for a.e. }t\in[0,T].
\end{equation}
To make Section \ref{sec1.2} rigorous, notice $u\in \TD$ we have
$$\int_\om \vph\d(\Delta u  )=-\int_\om\gr u\cdot \gr \vph \d x \,\text{ for any }\vph\in W^{1,p}(\om).$$
Particularly, taking $\vph\equiv 1$ gives
   $\int_\om \d (\Delta u)=0$, so  we have
$$\|(\Delta u)^+\|_{\mathcal{M}(\om)}=\|(\Delta u)^-\|_{\mathcal{M}(\om)}=\frac12\|\Delta u\|_{\mathcal{M}(\om)}.$$
Since
\begin{equation*}
\|(\Delta u)^-\|_{L^1(\om)} = \int_\om
(\Delta u)^-\d x \le \int_\om e^{(\Delta u)^-}\d x \le \int_\om e^{-(\Delta u)_{\|}^++(\Delta u)^-}\d x= \phi(u) \le \phi(u^0)
\end{equation*}
we know
\begin{equation}\label{psi00}
 (\Delta u)^-\ll\L^d\quad \text{ for a.e. }t\in[0,T],\quad  \|\Delta u\|_{\mathcal{M}(\om)}\le 2\phi(u^0),
\end{equation}
so in Definition \eqref{psi}, we can just take
\begin{equation}\label{consC}
 C_*:=2\phi(u^0)+1
\end{equation}
 and
\begin{equation}\label{psi0}
\psi(u(t))\equiv 0\equiv \pd\psi(u(t)).
\end{equation}
{\blue The invariant ball introduced by $\psi$ is similar to the idea of a-priori assumption method in PDE. We first obtain the solution in some invariant ball $\|\Delta u\|_{\mathcal{M}}\leq C_*$, then we prove the invariant ball is not artificial by showing the solution truly locates within the ball $\|\Delta u\|_{\mathcal{M}}\leq C_*-1.$
Noticing also that if $v\in D(\psi)$, $\psi(v)=0$, so
we can rewrite EVI \eqref{0927-32} as
$$\frac12\frac\d{\d t}\|u(t)-v\|^2\le \phi(v)-\phi(u(t)), \quad \text{for a.e. } t>0,\,\forall v\in D(\phi+\psi).
$$}

Next, we need to show that $u_t\in L^\8(0,T;L^2(\om))$. From Theorem \ref{th-vi} we know that $t\mapsto u(t)$ is locally Lipschitz in
$(0,T)$, i.e. for any $t_0>0$ there exists $L=L(t_0)>0$ such that
$$\|u(t_0+\vep)-u(t_0)\|_{L^2(\om)}\le L(t_0)\vep \qquad \text{for all } {\blue \vep\in [0, T-t_0]}.$$
The key point is to obtain a uniform bound for $L(t_0)$ for arbitrary $t_0\ge 0$. Since $u(t)$ is the variational solution satisfying \eqref{vi},
 taking $v=u(t_0)$ in \eqref{vi} gives
\begin{align*}
\frac12&\frac\d{\d t}\|u(t_0)-u(t)\|_{L^2(\om)}^2 \le \phi(u(t_0))-\phi(u(t)) \le \lg \xi ,u(t_0)-u(t)\rg
\end{align*}
for any $\xi \in \pd \phi(u(t_0))$. In particular, by \cite[Proposition 1.4.4]{AGS}, we have
\begin{equation}\label{Q1}
|\pd \phi|(u(t_0))=\min \{\|\xi\|_{H'}; \xi\in\pd\phi(u(t_0))\}.
\end{equation}
Hence taking $\xi$ as the elements of minimal dual norm in $\pd \phi (u(t_0))$ implies
\begin{align*}
\frac12\frac\d{\d t}\|u(t_0)-u(t)\|_{L^2(\om)}^2 &\le\phi(u(t_0))-\phi(u(t))\\
& \le \| \xi\|_{L^2(\om)'} \|u(t_0)-u(t)\|_{L^2(\om)}\\
& \le |\pd \phi|(u(t_0))\|u(t_0)-u(t)\|_{L^2(\om)}.
\end{align*}
Furthermore,
since $t\mapsto \|u(t_0)-u(t)\|_{L^2(\om)}$ is locally Lipschitz, hence differentiable
for a.e. $t$, we have
\begin{equation}
\frac\d{\d t}\|u(t_0)-u(t)\|_{L^2(\om)}\le |\pd \phi|(u(t_0)) \le
|\pd \phi|(u^0) \qquad \text{for a.e. }t>0,
\label{lip}
\end{equation}
where we have used \eqref{Dphi-dec} in the last inequality. {\blue From \eqref{Q1}, $|\pd\phi|(u^0)$ is just the subdifferential of $\phi(u^0)=\int_\om e^{-(\Delta u^0)_\|}\d x$. We know if the Gateaux-derivative of $\phi(u^0)$ exists in some dense set of $D(\phi)$, then the subdifferential of $\phi(u^0)$ is single-valued. Therefore direct calculation gives $\partial \phi (u^0)=\Delta e^{-(\Delta u^0)_\|}$ and $|\pd \phi|(u^0)=\|\Delta e^{-(\Delta u^0)_\|}\|_{L^2(\om)}$.}
Thus the function $t\mapsto \|u(t_0)-u(t)\|_{L^2(\om)}$
is globally Lipschitz with Lipschitz constant less than $|\pd \phi|(u^0)$, which is independent of $t_0$.
From \cite[Theorem 1.17]{Barbu2010}, $u$ is differentiable a.e. in $[0,T]$ w.r.t $H$, and belongs to $W^{1,\infty}([0,T];H)$.
Hence we know
$$\left\|\frac{u(t_0)-u(t_0+\vep)}{\vep}\right\|_{L^2(\om )}\le  |\pd \phi|(u^0)$$
Thus for a.e. $t$ we have
\begin{equation*}
\dfrac{u(t+\vep)-u(t)}\vep \in L^2(\om),\qquad \left\|\dfrac{u(t+\vep)-u(t)}\vep\right\|_{L^2(\om)}\le |\pd \phi|(u^0),
\end{equation*}
and the sequence of difference quotients $\dfrac{u(t+\vep)-u(t)}\vep$ is uniformly bounded in $L^2(\om)$. Since $u$ is differentiable a.e. in $[0,T]$ and the derivative is unique, define $u_t(t):=\lim_{\vep\to 0}\dfrac{u(t+\vep)-u(t)}\vep$. Consequently,
 \begin{equation}\label{ut925}
 \|u_t\|_{L^\8(0,T;L^2(\om))}\le |\pd \phi|(u^0)=\|\Delta e^{-(\Delta u^0)_\|}\|_{L^2(\om)} .
 \end{equation}

 Finally, from
 $$\frac12\frac\d{\d t}\|u(t)-v\|_{L^2(\om)}^2=\lg u_t(t), u(t)-v\rg,$$
 we obtain \eqref{vi35}.
 \end{proof}

\subsection{Existence of strong solution}
After establishing the regularity of variational inequality solution in Section \ref{Sec3.1}, we start to prove the variational inequality solution is also a strong solution. We first clarify the definition of strong solution, which has a latent singularity.
\begin{defn}\label{defstrong}
Given initial datum $u^0 \in H$ such that $\phi(u^0)<+\8$, we call function
 $$u\in L^\8([0,T];\tdV )\cap C^0([0,T];H),\quad u_t\in L^\8([0,T];H)$$
 a strong solution to \eqref{maineq} if $u$ satisfies
 \begin{equation}\label{tt57}
         u_t=\Delta( e^{-(\Delta u)_{\|}})
        \end{equation}
        for a.e. $(t,h)\in[0,T]\times \I$, where $(\Delta u)_{\|}$ is the absolutely continuous part of $\Delta u$ in the decomposition \eqref{decom}.
\end{defn}
{\blue
\begin{remark}
The equation \eqref{tt57} holds for a.e. $(t,h)\in[0,T]\times \I$ in the sense that
\begin{equation}
 \int_\I \big[u_t(t) -\Delta e^{-(\Delta u(t))_\|} \big] \vph\d x =0\qquad \fal\vph\in C_c^\8(\I)
\end{equation}
for a.e. $t\in[0,T]$.
\end{remark}
}
Let $\vph\in C^{\8}_c(\I)$ be given. We prove the sub-differential of functional $\phi$ is single-valued {\blue along EVI solution $u$}.  The idea of proof is to test \eqref{vi35} with $v:=u\pm\vep \vph$ and then take limit as $\vep\to 0.$
Recall the space notation $H$ in \eqref{Hnote}
\begin{equation*}
H=\left\{u\in L^2(\om):\int_\om u\d x=0\right\}.
\end{equation*}
Let us state existence result for strong solution as follows.

\begin{thm}\label{mainth1}
  Given $T>0$, initial datum initial datum $u^0 \in H$ such that $\phi(u^0)<+\8$, then {\blue EVI} solution $u$ obtained in Corollary \ref{cor10} is  also a strong solution
  to \eqref{maineq}, i.e.,
        \begin{equation}\label{main21_3}
         u_t=\Delta( e^{-(\Delta u)_{\|}})
        \end{equation}
        for a.e. $(t,h)\in[0,T]\times \I$.
        Besides, we have
        $$\Delta( e^{-(\Delta u)_{\|}})\in L^\8([0,T];H)$$
         and the following two dissipation inequalities
        $$
         \phi(u(t))=\int_\I e^{-(\Delta u(t))_{\|}}\d x\leq \phi(u^0), \quad t\geq 0,
        $$
         \begin{equation}\label{dissiE}
           E(u(t)):=\frac{1}{2}\int_\I \big[\Delta( e^{-(\Delta u)_{\|}})\big]^2 \d x\leq E(u^0), \quad t\geq 0,
         \end{equation}
        where $(\Delta u)_{\|}$ is the absolutely continuous part of $\Delta u$ in the decomposition \eqref{decom}.
\end{thm}
\begin{proof}

 Step 1. Integrability results.

First from \eqref{phi0n}, we know
\begin{equation}\label{L1}
e^{-(\Delta u(t))_{\|}}\in L^1(\I).
\end{equation}
Since $\vph\in C_c^{\8}(\om)$
 we also know
\begin{equation}
\label{well-defined}
e^{-(\Delta u(t))_\|-\vep \Delta \vph}\in L^1(\I)
\end{equation}
for all sufficiently small $\vep$.

{Step 2. Testing with $v=u(t)\pm\vep \vph$.}

{\blue First we show $v\in D(\phi+\psi)$. Since $\vph \in C_c^\8$, it is sufficient to show $v\in D(\psi)$ for $\vep$ small enough. Indeed, from \eqref{psi00} we know $\|\Delta u\|_{\mathcal{M}}\leq 2\phi(u^0)=C-1.$ Hence we choose $\vep$ small enough such that $\vep\leq \frac{1}{2\|\vph\|_{W^{2,\8}}}$, which implies $\|v\|_{\mathcal{M}}\leq 2\phi(u^0)+\frac{1}{2}< C$ and $\psi(v)=0.$ }

Plugging $v=u(t)+\vep \vph$ in \eqref{vi35} gives
\begin{equation}
\lg u_t(t),\vep\vph\rg+\phi(u(t)+\vep \vph)-\phi(u(t))\ge 0.
\label{9vi-2}
\end{equation}
Direct computation shows that
\begin{align*}
\phi(u(t)+\vep \vph)-\phi(u(t)) &= \int_\I\Big[e^{-(\Delta u(t))_\|-\vep\Delta\vph}-e^{-(\Delta u(t))_\|}\Big]\d x\\
&= \int_\I e^{-(\Delta u(t))_\|-\vep\Delta\vph}\Big(1-e^{\vep\Delta\vph}\Big)\d x\\
&\le -\int_\I e^{-(\Delta u(t))_\|-\vep\Delta\vph}\Big(\vep\Delta\vph\Big)\d x,
\end{align*}
where we used {\blue $1-e^x\leq -x$ for all $x\in \mathbb{R}.$}
This, together with \eqref{9vi-2}, gives
\begin{equation}
\lg u_t(t),\vep\vph\rg-\int_\I e^{-(\Delta u(t))_\|-\vep\Delta\vph}\Big(\vep\Delta\vph\Big)\d x\ge 0.\label{9vi-3}
\end{equation}
To take limit in \eqref{9vi-3}, we claim
\begin{align}
&\lim_{\vep\to 0}\int_\I e^{-(\Delta u(t))_\|-\vep\Delta\vph}\Delta \vph \d x =
\int_\I e^{-(\Delta u(t))_\|}\Delta\vph \d x.
\label{c-3}
\end{align}

{\em Proof of \eqref{c-3}.} First we have
$$ e^{-(\Delta u(t))_\|-\vep\Delta\vph} \Delta \vph \to e^{-(\Delta u(t))_\|} \Delta \vph  \quad  \text{ a.e. on }\om .$$
Then by \eqref{well-defined} we can see
$$\int_\om e^{-(\Delta u(t))_\|-\vep\Delta\vph} \Delta \vph \d x <+\8. $$
Thus by dominated convergence theorem we infer \eqref{c-3}.

\bigskip

Now we can divide by $\vep>0$ in \eqref{9vi-3} and take the limit $\vep\to0^+$ to obtain
\begin{align*}
&\lg u_t(t),\vph\rg-\lim_{\vep\to0^+} \int_\I e^{-(\Delta u(t))_\|-\vep\Delta\vph}\Delta\vph\d x \\
&=\lg u_t(t),\vph\rg -
 \int_\I e^{-(\Delta u(t))_\|}\Delta\vph \d x\geq 0.
\end{align*}
Repeating the above arguments with $v=u(t)-\vep \vph$ gives
\begin{equation*}
\lg u_t(t),\vph\rg -
 \int_\I e^{-(\Delta u(t))_\|}\Delta\vph \d x\le0.
\end{equation*}
Thus we finally have
\begin{equation}
 \int_\I \bigg[u_t(t)\vph -e^{-(\Delta u(t))_\|}\Delta\vph \bigg] \d x =0\qquad \fal\vph\in C_c^\8(\I).
\end{equation}

 Therefore
$u_t(t) -\Delta e^{-(\Delta u(t))_\|}=0$ in $C_c^\8(\I)'$. From the Radon-Nikodym theorem, we also know
$u_t =\Delta e^{-(\Delta u(t))_\|}$ for a.e. $(t,x)\in[0,T]\times\I.$

   Finally, we turn to verify \eqref{dissiE}.
    Combining \eqref{main21_3} and \eqref{ut925}, we have the dissipation
    {law}
  \begin{equation}\label{tE}
    E(u(t))=\frac{1}{2}\|u_t(t)\|_H^2=\frac{1}{2}\|\Delta e^{-(\Delta u(t))_\|}\|_H^2\leq \frac{1}{2}E(u^0),
  \end{equation}
   where $E(u(t))=\frac{1}{2}\int_\I \big[\Delta e^{-(\Delta u(t))_\|}\big]^2 \d x$ defined in \eqref{dissiE}.
  Hence the dissipation inequality \eqref{dissiE} holds and we completes the proof of Theorem \ref{mainth1}.
\end{proof}

\section*{Acknowledgment}
We would like to thank the support by the National Science Foundation under Grant No. DMS-1812573 and KI-Net RNMS11-07444.
XYL was partially supported by Lakehead University fundings 10-50-16422410 and
10-50-16422409, and NSERC Discovery Grant 10-50-16420120. Part of this work was done when XYL was a postdoc
at McGill University.

\end{document}